\documentclass[11pt]{article}
\usepackage[utf8]{inputenc}
\usepackage[T1]{fontenc}
\usepackage[english]{babel}
\usepackage{amsmath}
\usepackage{amsfonts}
\usepackage{amssymb}
\usepackage{amsthm}
\usepackage{listings}
\usepackage{xcolor}
\usepackage{cite}
\usepackage{pgf,tikz,pgfplots}
\usepackage{subcaption,ragged2e}
\usepackage{enumitem}
\pgfplotsset{compat=1.15}
\pgfplotsset{ticks = none}
\usepackage{mathrsfs}
\usetikzlibrary{arrows}
\usepackage{verbatim}
\usepackage{mathtools}
\usepackage{multicol} 
\usepackage{hyperref}
\usepackage{comment}
\usepackage{pdflscape}
\usepackage{afterpage}
\usepackage[left=1in, right=1in, bottom = 1in, top = 1in]{geometry}
\newcommand*{\N}{\mathbb{N}}
\newcommand*{\Z}{\mathbb{Z}}

\newcommand*{\R}{\mathbb{R}}

\newcommand*{\msf}{\mathsf{m}}

\renewcommand{\phi}{\varphi}
\renewcommand{\epsilon}{\varepsilon}
\renewcommand{\theta}{\vartheta}

\newtheorem{lem}{Lemma}[section]
\newtheorem{prop}{Proposition}[section]
\newtheorem{defn}{Definition}[section]
\newtheorem{thm}{Theorem}[section]
\newtheorem{cor}{Corollary}[section]
\newtheorem{example}{Example}[section]
\newtheorem{rem}{Remark}[section]
\newtheorem{ass}{Assumptions}[section]

\makeatletter
\newcommand{\labeltext}[2]{%
  \@bsphack
  \MakeLinkTarget*{#1}%
  \def\@currentlabel{#1}{\label{#2}}%
  \@esphack
}
\makeatother

\title{Locally Markov walks on finite graphs}
\author{Robin Kaiser, Lionel Levine, Ecaterina Sava-Huss}

\date{\today}

\begin{document}
\maketitle

\begin{abstract}
Locally Markov walks are natural generalizations of classical Markov chains, where instead of a particle moving independently of the past, it decides where to move next depending on the last action performed at the current location. We introduce the concept of locally Markov walks and we describe their stationary distribution and recurrent states, and we prove several properties such as irreducibility and ergodicity. For a particular  locally Markov walk - the uniform unicycle walk on the complete graph - we investigate the mixing time and we prove that it exhibits cutoff.
\end{abstract}


\section{Introduction}
\textcolor{black}{Consider the lattice $\mathbb{Z}^2$, where each vertex is equipped with an arrow pointing toward one of its four neighboring vertices. A particle then moves according to the following rule: when the particle is at site $x$, we first rotate the arrow at $x$ clockwise to the next neighbor of $x$, and then move the particle along this updated arrow. This deterministic process is known as a rotor walk. If, instead, the updating of arrows is randomized — depending only on the current orientation of the arrow — then we obtain a stochastic generalization of rotor walks, that we call \textit{locally Markov walks}.}

The aim of the current work is to introduce and investigate \emph{locally Markov walks on  graphs} $G=(V,E)$ as stochastic processes that obey a weak form of the Markov property: if the walker's current position is $x\in V$, then the distribution of the next step depends on the previous exit from $x$, thus interpolating between random walks (where the exits are i.i.d.) and their deterministic analogues rotor walks (in which the exits are periodic). For such processes, we describe the stationary distribution, eigenvalues, recurrent states, irreducibility, and mixing times. While locally Markov walks admit a natural definition on infinite graphs, in this work we confine our analysis to finite graphs, which enables the derivation of general results.

For locally Markov walks, a natural question is to determine which typical properties of Markov chains carry over to this new setting and which are lost, given that these processes satisfy only a weak form of the Markov property. A fundamental feature of Markov chains is the existence and uniqueness of a stationary distribution together with convergence to stationarity; in Section \ref{sec:st-distr}, we establish the existence of stationary distributions for locally Markov walks and in Proposition \ref{prop:ergodic} show that ergodic properties are likewise preserved. \textcolor{black}{We also show that, in the limit, the one-step transition probabilities of a locally Markov walk behave like single steps of a Markov chain.}

{\color{black}
To state the main results, we introduce locally Markov walks and their local chains; see Section \ref{sec:local-total-walks} for further details. Let $G=(V,E)$ be a finite, strongly connected, directed graph, and we denote by $N_x$ the neighbours of $x$ in $G$. Consider a discrete-time random walk $(X_n)_{n\in\mathbb{N}}$ on $G$, which is not assumed to be a Markov chain. We call $(X_n)_{n\in\mathbb{N}}$ a locally Markov walk if, for each vertex $x\in V$, the sequence of successive exits from $x$ forms a Markov chain.
For $x\in V$, let $(\tau_k^x)_{k\in\mathbb{N}}$ denote the sequence of successive visits to $x$, defined by $\tau_0^x:=0$ and, for $k\geq 1$,
\begin{equation}\label{eq:exit-times}
\tau_k^x := \min\{j > \tau_{k-1}^x : X_j = x\}.
\end{equation}
The stopping time $\tau_k^x$ represents the $k$-th visit of $(X_n)_{n\in\mathbb{N}}$ to $x$. We then define
\begin{equation}\label{eq:local-chains}
L^x(k) := X_{\tau_k^x+1},
\end{equation}
so that $L^x(k)$ is the location of the walk immediately after its $k$-th visit to $x$. The process $(L^x(k))_{k\in\mathbb{N}}$ is called the local chain at $x$. Thus, $(X_n)_{n\in\mathbb{N}}$ is a locally Markov walk on $G$ if, for every $x\in V$, the associated local chain $(L^x(k))_{k\in\mathbb{N}}$ is a Markov chain. By enlarging the state space from $V$ to $V \times V^{V}$ and recording, in addition to the current position of the particle, the most recent exit from each vertex, one obtains a Markov chain on $V \times V^{V}$, referred to as the total chain associated with the locally Markov walk $(X_n)_{n\in\N}$. The total chain is uniquely determined by the collection of local chains.
We summarize below the main results of the paper.

\begin{thm}\label{thm:erg}
Let $(X_n)_{n\in\N}$ be a locally Markov walk on a finite, strongly connected graph $G=(V,E)$. Assume that all local chains have strictly positive transition probabilities on $N_x$ for all $x\in V$, and let $Q=(q_x(y))_{x,y\in V}$ be the stochastic matrix indexed over $V\times V$, whose row at index $x\in V$ is the stationary distribution of the local chain at $x\in V$. 
\begin{enumerate}[leftmargin=\parindent,align=left,labelwidth=\parindent,labelsep=0pt]
\item[(1)] (Proposition \ref{thm:stat-dist})  The stationary distribution of the total chain associated with  $(X_n)_{n\in\N}$ is obtained from the $q$-weighted spanning unicycles of $G$\label{thm3}.
\item[(2)] (Proposition \ref{prop:single-step-conv})  For all $x,y\in V$ we have
$\lim_{n\rightarrow\infty}\mathbb{P}(X_{n+1}=x\vert X_n=y)=q_y(x).$\label{thm2}
\end{enumerate}
\end{thm}
}
In Section \ref{sec:unicycle-walk}, we analyze the uniform unicycle walk, defined as the Markov chain induced by the simple random walk in which, on top of the particle location, we also keep track of all its previous past actions at every vertex. On the complete graph, we show that the uniform unicycle walk exhibits an abrupt convergence to stationarity, i.e., the cutoff phenomenon, and we compute the spectrum of its transition matrix; see Section \ref{sec:unicycle-walk} for details. We establish the following result.

\begin{thm}\label{thm:cut-off}
The uniform unicycle walk on the complete graph $K_m$ with $m$ vertices exhibits cutoff at time $m\log(m)$, for 
 $m\in\N$. 
\end{thm}
\textbf{Outline of the paper.} 
In Section \ref{sec:local-total-walks}, we introduce locally Markov walks and show how some known models can be described as locally Markov walks. We also introduce the most important tool for studying locally Markov walks, the total chain associated to a locally Markov walk. This construction embeds a locally Markov walk into an ordinary Markov chain framework. In Section \ref{sec:rec}, we use the total chain description to study the recurrent states of locally Markov walks, to identify the stationary distribution, and to prove Theorem \ref{thm:erg}.
In Section \ref{sec:unicycle-walk}, we analyze the total chain of a simple random walk, derive the eigenvalues of its transition matrix for regular graphs, and establish the cutoff phenomenon for the total chain on complete graphs, as stated in Theorem \ref{thm:cut-off}.

\subsection{Related works}\label{sec:related}

\textbf{Rotor walks.} Our primary motivation for introducing and studying random walks with locally Markovian step distributions was to generalize the rotor walk model, independently introduced in \cite{rotor1,rotor2,rotor3}. Several results in Section \ref{sec:local-total-walks} were inspired by analogous results for rotor walks. In particular, the total chain construction for locally Markov walks, as defined in Definition \ref{def:totalchain}, corresponds directly to the framework used to define rotor walks.
For rotor walks on finite graphs, the recurrent rotor configurations correspond to unicycle configurations \cite[Corollary 3.5]{lev-hol-propp:rotor}. Hence, Proposition \ref{prop:recurrentstates} extends this result to the broader setting of locally Markov walks.

\textbf{Excited random walks.} Excited random walks on $\mathbb{Z}^d$ are processes in which the walker has a directional bias upon its first visit to a vertex, while on subsequent visits it chooses a neighbor uniformly at random \cite{excited-walk}. These walks have been further generalized to multi-excited or cookie random walks, in which each vertex is equipped with a stack of cookies. Each cookie contributes to a bias on the walker’s next step, and once all cookies at a vertex are exhausted, the walker moves uniformly at random. Cookie random walks have been studied with different types of stacks present at every vertex, including random but finite stacks \cite{cookie}, infinite periodic stacks \cite{periodic-cookie}, and stacks with finitely many types of cookies arranged in a Markovian way \cite{markovian-cookie}. 
Excited random walks  fall within the framework of locally Markov walks, whereas cookie random walks can be represented as locally Markov walks via a multiple-edge construction as described in Remark \ref{rem:multipleedge}.
See the survey \cite{excited-survey} for a collection of open problems on excited walks.

\textbf{Reinforced random walks.} Reinforced random walks are random walks in which the weight of an edge $e$ at time $n$ depends on the number of times that edge has been traversed up to time $n$. In \cite{davis}, it is shown that reinforced random walks on $\Z$ is either recurrent or it has finite range. Reinforced random walks on trees have been investigated in Pemantle \cite{pemantle}.
Reinforced walks do not fall within the class of locally Markov walks; however, they are closely related processes in which the current behaviour is influenced by the past behaviour at the present location.

\textbf{Tree walk.} In \cite{markov-tree}, the author introduced a random walk on the set of spanning trees of a finite graph, where the root of the tree performs a simple random walk at each time step and the edge set is updated accordingly to preserve the spanning-tree structure. This “tree random walk’’ was originally developed as a tool for proving the Markov chain tree theorem. Its construction is closely related to the total chain associated with a simple random walk. In Section \ref{sec:unicycle-walk}, we make use of this connection to extend the eigenvalue results of \cite{atha} for the tree random walk to the case of the uniform unicycle walk.

\section{Locally Markov walks}\label{sec:local-total-walks}

Let $G=(V,E)$ be a \textcolor{black}{locally finite, strongly connected}, directed graph with vertex set $V$ and edge set $E$, and denote by $"\rightarrow"$ the neighborhood relation in $G$. For an edge $(x,y)\in E$ we write sometimes $x\rightarrow y$, and for $x\in V$, we write $N_x$ for the set of neighbors of $x$, that is $N_x=\{y\in V:\ x \rightarrow y\}$. 
Recall the definition of the random process $(X_n)_{n\in\N}$ (not necessary Markovian), of the exit times from Equation \eqref{eq:exit-times}, and of the local chain $L^x$ at $x\in V$ as defined in \eqref{eq:local-chains}.
\begin{defn}\label{defn:locally-markov}
A random walk $(X_n)_{n\in\N}$ on $G$ is called a \emph{locally Markov walk}, if the process $(L^x(k))_{k\in\N}$ is a Markov chain for every $x\in V$.
\end{defn}
The initial state $(L^x(0))_{x\in G}$ of the locally Markov walk may be fixed or random.
A locally Markov walk satisfies a weaker Markov property: unlike a Markov chain, where the next step depends solely on the current state, it also depends on the action taken the last time the walker was at the current location.

\begin{example}\normalfont
\textbf{Excited random walk \cite{excited-walk} on $\Z$.} An excited walk is a random walk that has a drift in some direction when first visiting a vertex, and moves uniformly at random to the neighbours on each subsequent visit. More precisely, the excited random walk on $\Z$ with drift of magnitude $\varepsilon>0$ to the right can be modeled as a locally Markov walk \textcolor{black}{as follows}. 
Let each $n\in\Z$ have the three neighbours $n+1,n-1$ and $n$ itself, i.e.~we introduce self-loops for technical purposes. The local chain at $n\in\Z$ has state space $\{n-1,n,n+1\}$ and the transition probabilities are 
\begin{align*}
p(n,n+1)=\frac{1+\varepsilon}{2},&&p(n,n-1)=\frac{1-\varepsilon}{2},&&p(n,n)=0,
\end{align*}
if the previous exit from $n$ was to $n$ itself and
\begin{align*}
p(n+1,n+1)=p(n+1,n-1)=p(n-1,n+1)=p(n-1,n-1)=\frac{1}{2},
\end{align*}
in all the other cases.
Then the locally Markov walk with the above local transition probabilities starting from vertex $0$ with initial states  $(L^n(0))_{n\in\mathbb{Z}}$ of each chain pointing to itself, i.e.~given by $L^n(0)=n$ for every $n\in \N$,
is an excited random walk with initial right drift. 
\end{example}

\begin{example}\label{example:cyclic-gr}\normalfont
\textbf{Rotor walks on cyclic graphs.} For $m\in \N$, denote by $C_m$ the cycle graph with vertex set $V=\{0,1,...,m-1\}$ and edge set given by
$E=\{(k,l):\ k=l+1 \text{(mod $m$) or }k=l-1\text{(mod $m$)}\}.$
Define a deterministic walk $(X_n)_{n\in\N}$ called \emph{rotor walk or rotor-router walk} on  $C_m$ as follows. Let $X_0=0$ and $\rho_0\in \prod_{x\in V}N_x$; we call $\rho_0$ the rotor configuration at time $0$. Given the location $X_n$ of the walker and the rotor configuration $\rho_n$ at time $n\geq 0$, one step of the walk consists of the following: the walker changes the rotor configuration $\rho_n(X_n)$ at the current location to the other neighbor of $X_n$, then the walker moves to this new neighbor. That is, at time $n+1$ the rotor configuration and the particle (walker) position are
\begin{align*}
&\rho_{n+1}(x)=\begin{cases}\rho_n(x),&x\neq X_n\\
N_x\setminus\{\rho_n(x)\},&x=X_n\end{cases},\\
&X_{n+1}=\rho_{n+1}(X_n).
\end{align*} 
The rotor walk $(X_n)_{n\in\N}$ is a deterministic example of a locally Markov walk, as the only information needed to perform a step is the current location of the walker and the rotor at this location. For each $\ell\in C_m$, the local chain $L^{\ell}$ at $\ell$ is a deterministic walk on $\{\ell-1,\ell+1\}$ with transition
$\begin{pmatrix}
0 & 1 \\
1 & 0
\end{pmatrix}$.
\end{example}

\begin{example}\label{ex:p-walk}\normalfont
\textbf{$p$-walks \cite{p-walk} on $\Z$} are locally Markov walks that generalize rotor walks, in which at each time step the probability of the rotor being broken is $p\in(0,1)$. So a $p$-walk evolves as a rotor walk with probability $p$, and with probability $1-p$ the rotor does not work and the walker moves to the neighbor the rotor is pointing at. Thus for each $\ell\in\Z$, the local chain $L^{\ell}$ at $\ell$ is a Markov chain with state space $\{\ell-1,\ell+1\}$ that stays at $\ell-1$ (respectively $\ell+1$) with probability $1-p$ and moves from $\ell-1$ to $\ell+1$ and from $\ell+1$ to $\ell-1$ with probability $p$. 
\end{example}

\begin{rem}\normalfont\label{rem:multipleedge}
\textbf{Random walks with hidden local memory.} Inspired by the study of hidden Markov chains, one can generalize the locally Markov walks, so that the local state remembers more than the last exit. Suppose that each $x\in V$ is endowed with a local countable state space $S_x$, and a hidden mechanism $H_x=(h_x(s,t))_{s,t\in S_x}$ which is a Markov chain on $S_x$ with transition probabilities $h_x(\cdot,\cdot)$. 
Furthermore, for every $x\in V$ a jump rule as a function $f_x:S_x\rightarrow\text{Prob}(N_x)$ is given, where $\text{Prob}(N_x)$ is the set of probability distributions on $N_x$. A hidden state configuration $\kappa$ is then a vector indexed over the vertices $V$ with $\kappa(x)\in S_x$ for all $x\in V$.
We can then define a random walk with hidden memory as a Markov chain $(X_n,\rho_n,\kappa_n)_{n\in\N}$ with the following transition rules: given the state $(x_0,\rho_0,\kappa_0)$ at time $0$, for every $n\geq 1$
\begin{align*}
\text{(i) } \kappa_{n+1}(x)=\begin{cases}
K_n, & \text{if }x=X_n\\
\kappa_n(x), &\text{else}
\end{cases},&&
\text{(ii) } \rho_{n+1}(x)=\begin{cases}
Y_n,&\text{if }x=X_n\\
\rho_n(x),&\text{else}
\end{cases},&&
\text{(iii) } X_{n+1}=Y_n,
\end{align*}
where, given $X_n=x$, $K_n$ is a random element of $S_x$ sampled from the distribution $h_x(\kappa_n(x),\cdot)$ independent of the past and $Y_n$ is a random neighbor of $x$ sampled from $f_{x}(K_n)$ independent of the past. Intuitively, at each time step the walker at $x$ first updates its internal state $s$ to state $t\in S_v$ with distribution $h_x(s,t)$, and then it jumps to a random neighbor of $x$ with distribution $f_{x}(t)$.

\textbf{Multiple edge construction of the hidden Markov model.} The random walk with hidden local memory can be seen as a locally Markov walk in the sense of Definition \ref{defn:locally-markov}, if we allow the graph to have multiple edges labeled by the internal states. We describe shortly this construction. 
Let $G'$ be the graph obtained from $G=(V,E)$ in the following way. The vertex set of $G'$ is $V$, and there is an edge from $x$ to $y$ labeled by $s\in S_x$ if $(x,y)$ is an edge in $G$; we denote this edge by $(x,y)_s$.
We define a local transition rule at $x$ for $y,z\in N_x$ and $s,t\in S_x$ by
$$\mathsf{m}'_x((x,y)_s,(x,z)_t):=h_x(s,t)f_x(t)(z),$$
where $h_x$ is the transition rule for the internal states at $x$ and $f_x$ is the jump rule in the hidden Markov model.
The locally Markov walk $(X'_n)_{n\in\N}$ on $G'$ with local transition rules $\mathsf{m}'_x$ for all $x\in V$ is a random walk with hidden local memory $(X_n,\rho_n,\kappa_n)_{n\in\N}$ in the sense that
$(X'_n)_{n\in\N}\overset{d}{=}(X_n)_{n\in\N}.$
\end{rem}
 
\begin{example}\normalfont
\textbf{Cookie random walks \cite{cookie-singh}.} An example of a random walk with hidden local memory is the cookie random walk, also sometimes referred to as multi-excited random walk. Intuitively, the cookie random walker first has to eat through a stack of cookies at each vertex before it starts to move uniformly at random. To model this as a locally Markov walk on $\Z$, first we fix some $M\in\N$ and numbers $p_1,p_2,...,p_M\in(1/2,1]$. For each $n\in\Z$, connect $n$ to $n+1$ via $M+1$ edges denoted by $e^{n,n+1}_i$ for $i\in\{1,...,M+1\}$, as well as to $n-1$ via $M+1$ edges $e^{n,n-1}_i$ for $i\in\{1,...,M+1\}$. Then the local transition probabilities of the chain at $n\in \Z$ are given by 
\begin{align*}
p(e^{n,n+1}_i,e^{n,n+1}_{i+1})=p(e^{n,n-1}_i,e^{n,n+1}_{i+1})=p_i,&&p(e^{n,n+1}_i,e^{n,n-1}_{i+1})=p(e^{n,n-1}_i,e^{n,n-1}_{i+1})=1-p_i,
\end{align*}
for all $1\leq i \leq M$ and
\begin{align*}
p(e^{n,n+1}_{M+1},e^{n,n+1}_{M+1})=p(e^{n,n-1}_{M+1},e^{n,n+1}_{M+1})=p(e^{n,n+1}_{M+1},e^{n,n-1}_{M+1})=p(e^{n,n-1}_{M+1},e^{n,n-1}_{M+1})=\frac{1}{2}.
\end{align*}
The locally Markov walk with these local transition probabilities starting at $0$ with initial states \textcolor{black}{$(L^n(0))_{n\in\Z}$} given by $L^n(0)=e^{n,n+1}_{1}$
for all $n\in\Z$ is then a cookie random walk on $\Z$.
A simple modification of the transition probabilities also allows us to treat cookie random walks with infinitely many periodic cookies \cite{periodic-cookie}, by defining the transition probability of the $M$-th cookie as
\begin{align*}
&p(e^{n,n+1}_M,e^{n,n+1}_1)=p(e^{n,n-1}_M,e^{n,n+1}_1)=p_M,&&p(e^{n,n+1}_M,e^{n,n-1}_1)=p(e^{n,n-1}_M,e^{n,n-1}_1)=1-p_M.
\end{align*}
The example presented above is a special case of the cookie random walk studied in \cite{cookie}, in which a more general distribution of cookies was allowed, which can not be captured using only finitely many edges in the multiple edge construction of locally Markov walks.
\end{example}
\subsubsection*{Total chain of a locally Markov walk}\label{subsec:totalchain}

A locally Markov walk is determined by its local chains. Suppose that for every $x\in V$ we are given the local chain $(L^x(k))_{k\in\N}$, i.e.~a Markov chain with state space $S_x\subset V\setminus\{ x \}$ and transition matrix indexed over $S_x$ and denoted $\mathsf{M}_x=(\mathsf{m}_x(y,z))_{y,z\in S_x}$, that is, for $y,z\in S_x$ it holds:
$$\mathsf{m}_x(y,z)=\mathbb P[L^x(k+1)=z|L^x(k)=y],\quad \text{for }k\in\N.$$ 
For the Markov chain with hidden local memory, the set $S_x$ of internal states associated to $x$ is an arbitrary countable state space, while from now on, $S_x$ denotes a subset of the vertex set $V$ of $G$.
By abuse of notation, we regard the matrices $\mathsf{M}_x$ \textcolor{black}{as indexed over the vertices $V$ of $G$} , by adding zero entries for all rows and columns corresponding to vertices not in $S_x$.
By enlarging the state space $V$ of a locally Markov walk to $V\times V^{V}$ and keeping track not only of the particles current location, but also of the last exit from each vertex which can be encoded by a configuration $\rho:V\to V$, one obtains a Markov chain defined uniquely by the local chains $\{L^x,\ x\in V\}$ (which are Markov chains) and their corresponding transition matrices $\{\mathsf{M}_x:\ x\in V\}$.

\begin{defn}[Total chain associated to a locally Markov walk]\label{def:totalchain}
A locally Markov walk $(X_n)_{n\in\N}$ on $G=(V,E)$ can be lifted to a Markov chain $(X_n,\rho_n)_{n\in\N}$ with state space $\mathcal{S}=V\times V^V$ and  transition matrix $P=(p((x,\rho),(y,\eta)))_{(x,\rho),(y,\eta)\in V\times V^V}$ whose entries, i.e.~the transition probabilities in one step, are given, for $(x,\rho),(y,\eta)\in \mathcal{S}$, by 
\begin{align*}
p((x,\rho),(y,\eta))=\begin{cases}\mathsf{m}_x(\rho(x),y),&\rho \triangle \eta=\{x\}\text{ and }\eta(x)=y,\\0,&\text{else}\end{cases},
\end{align*}
with $\rho\triangle\eta=\{x\}$ indicating that $\rho$ and $\eta$ agree at all vertices except $x$. We refer to $(X_n,\rho_n)_{n\in\N}$ as the total chain associated with the locally Markov walk  $(X_n)_{n\in\N}$.
\end{defn}
\begin{rem}\normalfont
We highlight the connection between the configurations $(\rho_n)_{n\in\N}$ of the total chain and the local chains $(L^x(k))_{k\in\N}$ for  $x\in V$. If $\tau^x_k$ is the time of the $k$-th visit to $x\in V$, then
$$\rho_{\tau^x_k+1}(x)=\rho_{\tau^x_k+2}(x)=...=\rho_{\tau^x_{k+1}}(x)=L^x(k).$$
The sequence $(\rho_n)_{n\in\N}$ encodes the last actions performed at all the vertices up to time $n\in\N$, while the local chain $L^x(k)$ at $x$ at time $k\in\N$ specifies the action taken upon the $k$-th visit to $x$.
\end{rem}
The \textbf{graph of the locally Markov walk $(X_n)_{n\in\N}$ on $V$} is also determined by the local chains and denoted by  $\mathsf{G}$. It has vertex set $V$ and edge set $\mathsf{E}=\{(x,y):\ x\in V,y\in S_x\}$. Note that if we are already given a graph $G=(V,E)$ and for all $x\in V$ we have $N_x=S_x$, then the graph of the locally Markov walk is actually $G$ itself, i.e.~$G=\mathsf{G}$. \textcolor{black}{From this point onward,  we assume that the underlying graph $G$ is finite with $|V|=m$ for some  $m\in\N$.}

Let for all $x\in V$, $q_x$ denote the stationary distribution of the local chain at vertex $x$. We enlarge the domain of $q_x$ from $S_x$ to all of $V$, by simply setting it to $0$ on all $y\notin S_x$.
Finally, we introduce the stochastic matrix $Q = (q_x(y))_{x,y \in V}$, whose $x$-th row is given by the stationary distribution vector $q_x$ of the chain $L^x$. 

\begin{ass}\label{assu}\normalfont
Throughout this paper, we assume that for every $x \in V$, the local chain $L^x$ with state space $S_x$ is irreducible and aperiodic, and that the graph $\mathsf{G}$, with vertex set $V$ and edge set $\mathsf{E}=\{(x,y):\ x\in V,\ y\in S_x\}$, is strongly connected \textcolor{black}{and finite. Moreover, we assume henceforth that $N_x = S_x$ for all $x \in V$, so that the graph $\mathsf{G}$ of the locally Markov walk coincides with the underlying graph $G$ on which the model is defined.}
\end{ass}

Under these assumptions, for each $x \in V$, the local chain $L^x$ admits a unique stationary distribution $q_x$, with $q_x(y) > 0$ for all $y \in S_x$. In particular, $q_x \mathsf{M}_x = q_x$ holds for all $x \in V$. Denote by $\pi:V\to (0,1)$ the unique, strictly positive stationary distribution of the stochastic matrix $Q$ which satisfies $\pi Q=\pi$.
The following observation is immediate. 

\begin{lem}
If $\mathsf{G}$ is strongly connected and all local chains are irreducible, then $Q$ is irreducible.
\end{lem}
\begin{proof}
Since $q_x(y)>0$ for all $y\in S_x$ by irreducibility, we have that $Q_{x,y}>0$ if and only if $(x,y)$ is an edge of $\mathsf{G}$. Since $\mathsf{G}$ is strongly connected this proves the claim.
\end{proof}

Since $(X_n,\rho_n)_{n\in\mathbb{N}}$ is a Markov chain on a finite state space, it is natural to investigate its transient and recurrent states, to characterize its stationary distribution and the rate of convergence thereto, and to determine whether it exhibits a cutoff phenomenon. This is the the focus of the next sections. 

\section{Recurrent states and the stationary distribution}\label{sec:rec}

\subsection{Unicycles and recurrent states}\label{sec:recurrent-states}

We show now that the recurrent states of the total chain $(X_n,\rho_n)_{n\in\N}$ associated to the locally Markov walk $(X_n)_{n\in \N}$ are the spanning unicycles of $\mathsf{G}$, and that the stationary distribution is supported on such spanning unicycles weighted accordingly.
A directed spanning tree $T$ of $\mathsf{G}$ rooted at $x\in V$ is a cycle free spanning subgraph of $\mathsf{G}$ such that for every $y\in V\setminus\{x\}$, there exists a unique directed path from $y$ to $x$ in $T$. We denote the set of spanning trees of $\mathsf{G}$ rooted at $x\in V$ by $\mathsf{TREE}_x(\mathsf{G})$. For $T\in \mathsf{TREE}_x(\mathsf{G})$, we define its weight to be
\begin{align*}
\psi(T)=\prod_{(y,z)\in T}q_y(z),
\end{align*}
where $q_y$ is the unique stationary distribution of the local chain $L^y$ at $y\in V$ with state space $S_y$.  
\begin{defn}[\textcolor{black}{q-weighted} spanning tree]
A random spanning tree $\mathcal{T}$ with distribution proportional to its weight $\psi(T)$, i.e.~$\mathbb{P}(\mathcal{T}=T)\sim\psi(T)$,
where $T$ is a spanning tree of $\mathsf{G}$, is called \textcolor{black}{q-weighted} spanning tree of $\mathsf{G}$.
\end{defn}
Here $\mathbb{P}(\mathcal{T}=T)\sim\psi(T)$ means that there exists a normalizing constant $Z\in\R$ such that
$\mathbb{P}(\mathcal{T}=T)=\frac{\psi(T)}{Z},$
 and this notation is used for the rest of the paper. The  constant $Z$ is defined as
\begin{align}\label{eq:normalizing}
Z=\sum_{x\in V}\sum_{T\in\mathsf{TREE}_x(G)}\psi(T).
\end{align}
A spanning unicycle $U$ of $\mathsf{G}$ is a spanning subgraph of $\mathsf{G}$ which contains exactly one directed cycle and for every vertex $x\in V$ there exists a unique path from $x$ to this cycle. Let us denote by $\mathsf{CYC}(\mathsf{G})$ the set of spanning unicycles of $\mathsf{G}$. Similarly as for trees, we define for $U\in\mathsf{CYC}(\mathsf{G})$ its weight as
\begin{align*}
\psi(U)=\prod_{(y,z)\in U}q_y(z).
\end{align*}
\begin{defn}[\textcolor{black}{q-weighted} spanning unicycle]
A random spanning unicycle $\mathcal{U}$ with distribution
$\mathbb{P}(\mathcal{U}=U)\sim\psi(U)$,
for all $U$ spanning unicycles of $\mathsf{G}$, is called \textcolor{black}{q-weighted} spanning unicycle of $\mathsf{G}$.
\end{defn}

Given a spanning tree $T$ rooted at a vertex $x \in V$, a spanning unicycle $U$ can be obtained by adding any outgoing edge from $x$ to the tree $T$.
For a rotor walk on a graph $G$, the recurrent configurations of the total chain are the spanning unicycles of $G$ in view of \cite[Theorem 3.8]{lev-hol-propp:rotor}. This is also true for the total chain associated to a locally Markov walk, whose local chains $\{L^x:\ x\in V\}$ are irreducible and aperiodic and the proof is similar to the rotor walk case. When saying that a configuration $\rho:V\rightarrow V$ is a spanning unicycle of the graph $G$, we mean that the set of edges $\{(x,\rho(x)):x\in V\}$ forms a spanning unicycle of $G$.

\begin{lem}\label{lem:rec_involution}
Consider  a locally Markov walk $(X_n)_{n\in\N}$ with corresponding total chain $(X_n,\rho_n)_{n\in\N}$ and suppose that the graph $\mathsf{G}$ induced by $(X_n)_{n\in\N}$ is strongly connected. For any $n\in \N$, if $\rho_n$ is a spanning unicycle and $X_n$ lies on the unique cycle of $\rho_n$, then almost surely $\rho_{n+1}$ is a spanning unicycle and $X_{n+1}$ lies on the unique cycle of $\rho_{n+1}$.
\end{lem}
\begin{proof}
Since $\rho_n$ is a unicycle configuration, the set of edges $\{\rho_n(x)\}_{x\neq X_n}=\{\rho_{n+1}(x)\}_{x\neq X_n}$ contains no directed cycles. However, in the configuration $\rho_{n+1}$ every vertex has outdegree $1$, hence it must contain a directed cycle. But any such cycle must contain the edge $\rho_{n+1}(X_n)$, hence this cycle is unique and $X_{n+1}=\rho_{n+1}(X_n)$ lies on it.
\end{proof}

\textbf{The time reversal of the total chain $(X_n,\rho_n)$} is denoted by $(\hat{X}_n,\hat{\rho}_n)_{n\in\N}$ and it is the Markov chain with transition probabilities given by
\begin{align*}
\hat{p}((x,\rho),(y,\eta))=\hat{\mathsf{m}}_y(\rho(y),x),
\end{align*}
for $\rho\triangle\eta=\{y\}$ and $y$ is the predecessor of $x$ in the unique cycle in $\rho$, and $\widehat{\msf}_y$ is the time reversal of the local chain at $y$ given by
\begin{align} \label{eq:tr-pb-reversed}
\widehat{\msf}_y(z,\textcolor{black}{x}):=\msf_y(x,z)\frac{q_y(x)}{q_y(z)},
\end{align}
for all $x,z\in N_y$. The reversed chain encodes a particle moving "backwards" in time, i.e.~if the location of the walk $(\hat{X}_n,\hat{\rho}_n)_{n\in\N}$ at time $n$ is $(x,\rho)$, 
where $\rho$ is a spanning unicycle and $x$ lies on the unique cycle of $\rho$, then at time $n+1$ the reversed walk
will move the particle from $x$ backwards  along the unique cycle of $\rho$ and then the configuration at the new position $y$ will be changed according to the transition probabilities defined in Equation \eqref{eq:tr-pb-reversed}.

\begin{rem}\normalfont
Under Assumptions \ref{assu}, the locally Markov walk  will eventually visit every vertex of $G$. Once all vertices have been visited, the configuration $\rho$ forms a spanning unicycle. Consequently, the recurrent configurations of the total chain are contained in the set of spanning unicycles. Therefore, it is sufficient to define the time reversal of the total chain only on this set.
\end{rem}

\begin{lem}\label{lem:rec_involution_reversal}
Let $(\hat{X}_n,\hat{\rho}_n)_{n\in\N}$ be the time reversal of the total chain $(X_n,\rho_n)_{n\in\N}$ associated to the local walk $(X_n)_{n\in\N}$.  If $\hat{\rho}_n$ is a spanning unicycle and  $\hat{X}_n$ lies on the unique cycle of $\hat{\rho}_n$, for  some $n\in\N$, then almost surely $\hat{\rho}_{n+1}$ is a spanning unicycle and $\hat{X}_{n+1}$ lies on the unique cycle of $\hat{\rho}_{n+1}$.
\end{lem}
\begin{proof}
It suffices to show that every directed cycle in $\hat{\rho}_{n+1}$ passes through $\hat{X}_{n+1}$. Suppose the contrary holds, that is, there is a directed cycle in $\hat{\rho}_{n+1}$ which does not pass through $\hat{X}_{n+1}$. Since $\hat{\rho}_{n+1}$ agrees with $\hat{\rho}_n$ everywhere except at $\hat{X}_{n+1}$, this same cycle also occurs in $\hat{\rho}_n$ and avoids $\hat{X}_{n+1}$. This is a contradiction, since the unique cycle in $\hat{\rho}_n$ passes through $\hat{X}_{n+1}$, and the claim follows.\end{proof}

The $n$-step transition probabilities $p^{(n)}$ and $\hat{p}^{(n)}$ of the total chain $(X_n,\rho_n)_{n\in\N}$ and of the reversed chain $(\hat{X}_n,\hat{\rho}_n)_{n\in\N}$ respectively, are related as follows.

\begin{lem}\label{lem:same_rec}
For all $n\in\N$ and all $(x,\rho),(y,\eta)\in \mathcal{S}$ we have $p^{(n)}((x,\rho),(y,\eta))>0$ if and only if $\hat{p}^{(n)}((x,\eta),(y,\rho))>0$.
\end{lem}
\begin{proof}
The base case $n=1$ follows from the definition of the time reversal chain and the general case follows easily by induction.
\end{proof}
\begin{prop}\label{prop:unicycle->rec}
Consider a spanning unicycle $\rho$ of $\mathsf{G}$ and a vertex $x\in V$ that lies on the unique cycle of $\rho$. Then $(x,\rho)$ is a recurrent state of the total chain $(X_n,\rho_n)_{n\in\N}$.
\end{prop}
\begin{proof}
By Lemma \ref{lem:same_rec}, the total chain and its time reversal have the same recurrent configurations. The total chain starting from $(x,\rho)$ must eventually reach a recurrent state $(y,\eta)$. By Lemma \ref{lem:rec_involution} the configuration $\eta$ must be a spanning unicycle and $y$ must lie on its unique cycle. Thus there exists $n\in\N$ such that $\hat{p}^{(n)}((y,\eta),(x,\rho))>0$, and since $(y,\eta)$ is recurrent for the total chain and thus also for its reversal, this implies that $(x,\rho)$ is recurrent for the time reversal and thus also the total chain, completing the proof.
\end{proof}
It remains to show that all recurrent configurations of the total chain are spanning unicycles.
\begin{prop}\label{prop:rec->unicycle}
Let $(x,\rho)$ be a recurrent configuration of the total chain $(X_n,\rho_n)_{n\in\N}$. Then $\rho$ is a spanning unicycle of $\mathsf{G}$ and $x$ lies on its unique cycle.
\end{prop}
\begin{proof}
Since the local chains $\{L^x,\ x\in V\}$ are irreducible and aperiodic, we eventually visit every vertex of $G$ almost surely. Assume that we start at $(x,\rho)$, visit all the vertices of $G$ and eventually return to $(x,\rho)$. Since every vertex in $\rho$ has outdegree $1$, it must contain at least one directed cycle. Choose one of these cycles and denote its vertices by $x_1,...,x_k$. Let us assume $x$ does not lie on the chosen cycle. For each $i$, the last time the walker was at $x_i$, it moved to $x_{i+1}$ and hence the edge $(x_i,x_{i+1})$ was traversed more recently than the edge $(x_{i-1},x_i)$. Carrying this argument around the cycle leads to a contradiction. Thus every cycle in $\rho$ must go through the vertex $x$, which implies that $\rho$ is a spanning unicycle and $x$ lies on its unique cycle.
\end{proof}
Combining Proposition \ref{prop:unicycle->rec} and Proposition \ref{prop:rec->unicycle} yields the following characterization of the recurrent states of the total chain.

\begin{prop}\label{prop:recurrentstates}
Given a locally Markov walk $(X_n)_{n\in\N}$ with associated total chain $(X_n,\rho_n)_{n\in\N}$, the recurrent states of the total chain are the configurations of the form $(x,\rho)$, where $\rho$ is a spanning unicycle of $\mathsf{G}$ and $x$ lies on the unique cycle of $\rho$.
\end{prop}



\begin{defn}[Set of unicycle configurations]
For a finite and \textcolor{black}{strongly connected} graph $G$, we define the set of unicycle configurations as 
\begin{align*}\mathsf{U}_G:=\{(x,\rho):\rho\in\mathsf{CYC}(G)\text{ and $x$ lies on the unique cycle in $\rho$}\}.\end{align*}
\end{defn}\label{def:unic}

\begin{rem}\normalfont
At first glance, it may seem plausible that the time reversal of the total chain associated to a locally Markov walk could itself be represented as the total chain associated to another locally Markov walk. However, this is not the case. In the reversed process, the particle moves backwards along the unique cycle of the current unicycle configuration. Determining the next step therefore requires knowledge of the entire cycle structure, not merely the most recent action at the current vertex. Consequently, the particle’s motion is not locally Markovian, and the time-reversed process does not constitute a locally Markov walk.
\end{rem}

\subsubsection*{Single step decomposition of the total chain}\label{sec:single-step-decomp}

Next we present a decomposition of the transition matrix of the total chain into a block-diagonal component and an orthogonal component. This relies on the observation that a single step of the total chain consists of two micro-steps: for a particle located at $x$, the first step involves examining the previous action performed at $x$ and determining the next move based on the local chain $L^x$. The second step involves moving the particle to the randomly chosen vertex: if the last exit from  $x$ was to $y$, then the particle at $x$ moves to $z$ with probability $\mathsf{m}_x(y,z)$. This motivates us to decompose the matrix $P$ into a block diagonal matrix $B_{\text{loc}}$, whose blocks correspond to the transition matrices of the local chains, and a permutation matrix $A_{\mathsf{CYC}}$, such that $P=B_{\text{loc}} A_{\mathsf{CYC}}$.
\begin{defn}[Unicycle permutation matrix]\label{def:perm_matrix}
For a finite, directed, and \textcolor{black}{strongly connected} graph $G$, we define the unicycle permutation of $G$ as the mapping $\zeta:\mathsf{U}_G\rightarrow\mathsf{U}_G$ given by
$\zeta(x,\rho)=(\rho(x),\rho)$.
Let also $A_{\mathsf{CYC}}\in\{0,1\}^{\mathsf{U}_G\times\mathsf{U}_G}$ be the matrix with entries
\begin{align*}A_{\mathsf{CYC}}((x,\rho),(y,\eta))=\begin{cases}1,&\zeta(x,\rho)=(y,\eta)\\0,&\text{else}\end{cases},\end{align*}
and we call $A_{\mathsf{CYC}}$ the unicycle permutation matrix of $G$.
\end{defn}
By definition, $\zeta$ moves the particle along the unique cycle in $\rho$ while keeping the directions of the edges unchanged. 
The set $\mathsf{U}_G$ represents the set of recurrent states of $(X_n,\rho_n)_{n\in\N}$ according to Proposition \ref{prop:recurrentstates}.
The lemma below shows that permuting the columns of the transition matrix $P$ of the total chain according to $A_{\mathsf{CYC}}$ yields a block-diagonal matrix.

\begin{lem}\label{lem:ss-decomp}
For the total chain $(X_n,\rho_n)$ on $G$ with transition matrix $P$, and transition probabilities of the local chains given by $\mathsf{M}_x=(\mathsf{m}_x(y,z))_{\textcolor{black}{y,z\in N_x}}$ for all $x\in V$,  for all $(x,\rho),(y,\eta)\in \mathsf{U}_G$ we have 
\begin{align*}
P A_{\mathsf{CYC}}^T((x,\rho),(y,\eta))=\begin{cases}\mathsf{m}_x(\rho(x),\eta(x)),&x=y\text{ and } \rho\setminus x=\eta\setminus x\\
0,&\text{else}
\end{cases}.
\end{align*}
\end{lem}
\begin{proof}
Let $(x,\rho),(y,\eta)\in\mathsf{U}_G$. A simple computation gives
\begin{align*}
PA_{\mathsf{CYC}}^T((x,\rho),(y,\eta))=p((x,\rho),\zeta(y,\eta)),
\end{align*}
where $\zeta$ is the unicycle permutation from Definition $\ref{def:perm_matrix}$. If we denote $\zeta(y,\eta)=(y',\eta')$, then the right-hand side of the above equation is non-zero if $y'\in N_x$, $\rho\setminus x=\eta'\setminus x$ and $\eta'(x)=y'$, which together with the definition of $\zeta$ proves the claim.
\end{proof}
We define the block diagonal matrix $B_{\text{loc}}$ as the matrix whose blocks are given by the local transition probabilities $\mathsf{m}_x(\cdot,\cdot)$ \textcolor{black}{as follows}:
$$B_{\text{loc}}=PA^T_{\mathsf{CYC}}.$$
\begin{cor}\label{prop:ss-decomp}
For the transition matrix $P$ of the total chain $(X_n,\rho_n)_{n\in\N}$ restricted to its recurrent states $\mathsf{U}_G$ it holds
\begin{align*}P=B_{\text{loc}} A_{\mathsf{CYC}}.\end{align*}
\end{cor}
This follows immediately from Lemma $\ref{lem:ss-decomp}$ by multiplying the above equation with $A_{\mathsf{CYC}}^T$ from the right, and using that $A^T_{\mathsf{CYC}}=A^{-1}_{\mathsf{CYC}}$.

\textcolor{black}{
We will see in Lemma \ref{lem:ev-constant} that the single-step decomposition provides a partial description of the eigenvectors of the uniform unicycle walk introduced in Section \ref{sec:unicycle-walk}.
}
\subsection{Stationary distribution}\label{sec:st-distr}

We describe here the stationary distribution of the total chain $(X_n,\rho_n)_{n\in\N}$ associated to a locally Markov walk $(X_n)_{n\in\N}$. Recall the definition of the matrix $Q$, whose rows are the stationary distributions of the local chains. For a configuration $\rho$ we define
$\rho\setminus x:=\{(y,\rho(y)):\ y\in V_x\},$
where $x\in V$ and $V_x=V\setminus\{x\}$. That is, $\rho\setminus x$ is the configuration obtained from $\rho$ by removing the outgoing edge from $x$. In view of Proposition \ref{prop:recurrentstates}, the set of recurrent states of the total chain is $\mathsf{U}_G$.
\textcolor{black}{The next result proves (1) in Theorem \ref{thm:erg}.}

\begin{prop}\label{thm:stat-dist} The stationary distribution $\mu: V\times V^V\to [0,1]$ of the total chain $(X_n,\rho_n)_{n\in\N}$ associated to the locally Markov walk $(X_n)_{n\in \N}$ is given by 
\begin{align*}
\mu(x,\rho)=\begin{cases}
\frac{1}{Z}\prod_{y\in V}q_y(\rho(y)),&\text{if }\rho\setminus x\in\mathsf{TREE}_x(G)\\
0,&\text{otherwise}
\end{cases},
\end{align*}
where $Z$ is the normalizing constant from \textcolor{black}{Equation \eqref{eq:normalizing}}.
\end{prop}
\begin{proof}
Assume that $(X_n,\rho_n)$ is $\mu$-distributed, and recall that its transition matrix is $P$, while the transition matrix of the local chain at $x\in V$ is $\mathsf{M}_x=(\mathsf{m}_x(y,z))_{y,z\in N_v}$. Let $(y,\eta)\in \mathsf{U}_G$, and we take another 
unicycle configuration $(x,\rho)\in \mathsf{U}_G$  such that $p((x,\rho),(y,\eta))>0$. Then  $x$ must be a neighbor of $y$. In fact, $x$ must be the unique neighbor of $y$ lying on the cycle in $\eta$ and satisfying $\eta(x)=y$. We denote this neighbor by $x_y$. Moreover, by the update rule of the total chain, we necessarily have  $\rho\setminus x_y=\eta\setminus x_y$. Finally, the value of $\rho$ at $x_y$ can be any of the neighbours of $x_y$. Since $(X_n,\rho_n)$ is $\mu$-distributed, factorizing with respect to the first step yields
\begin{align*}
\mathbb{P}_\mu\Big[(X_{n+1},\rho_{n+1})=(y,\eta)\Big]&=\sum_{(x,\rho)}\mu(x,\rho)p((x,\rho),(y,\eta))\\&=\frac{1}{Z}\prod_{v\in V_{x_y}}q_v(\eta(v)) \sum_{z\in  N_{x_y}}q_{x_y}(z)\mathsf{m}_{x_y}(z,y)=\frac{1}{Z}\prod_{v\in V}q_v(\eta(v))=\mu(y,\eta).
\end{align*}
Since $\mu(y,\eta)>0$ for all unicycle configurations $(\eta,y)\in\mathsf{U}_G$ and $\sum_{(y,\eta)}\mu(y,\eta)=1$,
it follows that $\mu$ is indeed the stationary distribution of $P$.
\end{proof}

\subsubsection*{Irreducibility of the total chain}

We show here that if a locally Markov walk $(X_n)_{n \in \mathbb{N}}$ has strictly positive local transition probabilities, then the corresponding total chain is irreducible. 

\begin{prop}\label{prop:irred}
For a locally Markov walk with strictly positive local transition probabilities — \textcolor{black}{ that is, for all $x\in V$ and all $y,z\in N_x$ we have $\mathsf{m}_x(y,z)>0$} - the corresponding total chain is irreducible.
\end{prop}
\begin{proof}
For $x,y\in V$, using the strict positivity assumption of the transition probabilities, it is clear that the location of the locally Markov walk $(X_n)_{n\in \N}$ can  be transitioned from $x$ to $y$ in a finite number of steps with positive probability. Therefore, it remains to show that the unicycle configuration of the total chain can also be transitioned into any other configuration with the same root in a finite number of steps with positive probability. 
For this, let $(y,\eta),(y,\eta')\in \mathsf{U}_G$, and suppose that $\eta$ and $\eta'$  differ only in a single vertex. We distinguish two cases.

If $\eta(y)\neq\eta'(y)$, then we can change the state of $\eta$ at $y$ by performing the following steps in the total chain: first perform a step from $y$ to $\eta'(y)$, which changes $\eta(y)$ to the desired neighbor $\eta'(y)$. Then from $\eta'(y)$, follow the unique directed path back to $y$. This way, the total chain will have moved from $(y,\eta)$ to $(y,\eta')$, and by the positivity assumption all steps have positive probability.

If, on the other hand, there exists some $x\neq y$ with $\eta(x)\neq\eta'(x)$, we can perform the following sequence of steps: denote the unique directed path from $x$ to $y$ in $\eta$ by $\gamma$ and the path from $\eta'(x)$ to $y$ by $\gamma'$. 
First, we traverse $\gamma$ in reverse order, moving the particle from $y$ to $x$ while reversing the direction of all arrows along $\gamma$.
Next we move from $x$ to $\eta'(x)$. Then we follow $\gamma'$ until we either reach $y$ or we intersect $\gamma$. 
Finally, we retrace $\gamma$ back to the last vertex before $x$, and from there, we follow the original order of $\gamma$ to return to $y$. After restoring the direction of the arrow at $y$ to its original orientation, we will have transitioned from $(y,\eta)$ to $(y,\eta')$ in a finite number of steps, where each step occurs with positive probability by assumption.
The proposition then follows by observing that any given unicycle configuration can be transformed into any other configuration by iteratively adjusting all mismatched arrows.
\end{proof}

The positivity assumption on the local transition probabilities can be slightly weakened.  The proof remains valid if we instead assume that the local chains are irreducible and that, for all $x\in V$ and $y\in N_x$ we have
$\mathsf{m}_x(y,y)>0$.
Under this assumption,  it is still possible to follow paths in the unicycle configuration in their given order, as well as to reverse the order of a given path. However, to provide a more concise formulation of the ergodicity of locally Markov chains, we will continue to use the positivity assumption on the local transition probabilities.

\begin{prop}\label{prop:ergodic}
Let $(X_n)_{n\in\N}$ be a locally Markov walk with local chains $(L^x(k))_{k\in\N}$ with strictly positive transition probabilities, for all $x\in V$. Then for all $f:V\rightarrow\R$ almost surely 
$$\lim_{t\to\infty}\frac{1}{t}\sum_{n=0}^t f(X_n)=\sum_{x\in V}f(x)\pi(x).$$
In particular, for $x\in V$ the proportion of visits to $x$ by $(X_n)_{n\in\N}$ until time $t$ converges to $\pi(x)$ as $t\to\infty$.
\end{prop}
\begin{proof}
If we write $N_t(x)$ for the number of visits to $x\in V$ by the locally Markov walk $(X_n)_{n\in \mathbb{N}}$ up to time $t\in\mathbb{N}$, it suffices to show for all $x\in V$ that, almost surely,
$$\lim_{t\rightarrow\infty}\frac{N_t(x)}{t}=\pi(x).$$
The proof relies on the observation that $X_{t+1}=L^{X_t}(N_t(X_t))$.
Let us now fix some $x\in V$, then
\begin{equation}\label{eq:thm1}
\begin{aligned}
\frac{1}{t}\big(N_t(x)-\delta_{x}(X_0)\big)&=\frac{1}{t}\sum_{n=0}^{t-1}\delta_x(X_{n+1})=\frac{1}{t}\sum_{n=0}^{t-1}\delta_x(L^{X_n}(N_n(X_n))\\
&=\frac{1}{t}\sum_{n=0}^{t-1}\sum_{y\in V}\delta_x(L^{X_n}(N_n(y))\delta_y(X_n)=\frac{1}{t}\sum_{y\in V}\sum_{n=0}^{t-1}\delta_x(L^{X_n}(N_n(y))\delta_y(X_n)\\
&=\frac{1}{t}\sum_{y\in V}\sum_{k=0}^{N_{t-1}(y)}\delta_x(L^y(k))=\sum_{y\in V}\frac{N_{t-1}(y)}{t}\frac{1}{N_{t-1}(y)}\sum_{k=0}^{N_{t-1}(y)}\delta_x(L^y(k)).
\end{aligned}
\end{equation}
Since all local transition probabilities are positive, and the underlying graph $G$ is strongly connected, we have $N_t(y)\rightarrow\infty$ as $t\rightarrow\infty$ for all $y\in V$. Using this, we obtain for all $y\in V$
$$\lim_{t\rightarrow\infty}\frac{1}{N_{t-1}(y)}\sum_{k=0}^{N_{t-1}(y)}\delta_x(L^y(k))=q_y(x).$$
By assumption, the total chain is finite and irreducible, so using the ergodic theorem for finite Markov chains, we have 
$$\lim_{t\rightarrow\infty}\frac{N_{t-1}(y)}{t}=f(y),$$
for some function $f:V\rightarrow(0,\infty)$ and for all $y\in V$. Reinserting this into Equation (\ref{eq:thm1}) and taking the limit as $t\to\infty$ on both sides yields
$$f(x)=\sum_{y\in V} f(y)q_y(x),$$
thus $f$ is an eigenvector of $Q$ with eigenvalue $1$. Since
$\sum_{y\in V}f(y)=1$, together with the uniqueness of the stationary distribution $\pi$ of $Q$ implies $f=\pi$, and this finishes the proof.
\end{proof}

\subsection*{Convergence of the single-step distribution}

According to Proposition \ref{prop:ergodic},  locally Markov walks inherit ergodic properties from the total chain.  In what follows, we establish a non-averaged analogue of Proposition \ref{prop:ergodic}, and we show that, in the long run, the single-step distributions of locally Markov walks resemble those of a Markov chain with transition matrix $Q$. We again assume that all local transition probabilities are strictly positive on $N_x$ for all $x\in V$.
As a simple consequence, from basic theory on finite state space Markov chains we obtain that
\begin{equation}
\begin{aligned}\label{eq:single-step}
\mathbb{P}(X_n=x)&=\sum_{\rho\in \mathsf{TREE}_x(G),y\in N_x}\mathbb{P}((X_n,\rho_n)=(x,\rho\cup(x,y)))\\
&\xrightarrow[]{n\rightarrow\infty}\sum_{\rho\in \mathsf{TREE}_x(G),y\in N_x}\mu(x,\rho\cup(x,y))=\pi(x).
\end{aligned}
\end{equation}
\textcolor{black}{The following proposition proves part (2) of Theorem \ref{thm:erg}.}
\begin{prop}\label{prop:single-step-conv}
Consider a locally Markov walk $(X_n)_{n\in\N}$ on $G$ with local chains $(L^x(k))_{k\in\N}$ that have strictly positive transition probabilities, for all $x\in V$.  Then for all $x,y\in V$ it holds
$$\lim_{n\to\infty}\mathbb{P}(X_{n+1}=x|X_n=y)=q_y(x).$$
\end{prop}
\begin{proof}
We have
\begin{align*}
\mathbb{P}(X_{n+1}=x|X_n=y)&=\frac{\mathbb{P}(X_{n+1}=x,X_n=y)}{\mathbb{P}(X_n=y)}\\
&=\frac{1}{\mathbb{P}(X_n=y)}\sum_{\rho\in \mathsf{TREE}_y(G),z\in N_y}\mathbb{P}(X_{n+1}=x,X_n=y,\rho_n=\rho\cup(y,z))\\
&=\frac{1}{\mathbb{P}(X_n=y)}\sum_{\rho\in \mathsf{TREE}_y(G),z\in N_y}\mathsf{m}_y(z,x)\mathbb{P}((X_n,\rho_n)=(y,\rho\cup(y,z))).
\end{align*}
From the ergodicity of the total chain we have
$$\lim_{n\rightarrow\infty}\mathbb{P}((X_n,\rho_n)=(y,\rho\cup(y,z)))=\mu(y,\rho\cup(y,z)),$$
which combined with Equation (\ref{eq:single-step}) implies
\begin{align*}
\lim_{n\rightarrow\infty}\mathbb{P}(X_{n+1}=x|X_n=y)&=\frac{1}{\pi(y)}\sum_{\rho\in \mathsf{TREE}_y(G),z\in N_y}\mathsf{m}_y(z,x)\mu(y,\rho\cup(y,z))\\
&=\frac{1}{\pi(y)}\sum_{\rho\in \mathsf{TREE}_y(G)}\mu(y,\rho\cup(y,x))=\frac{1}{\pi(y)}q_y(x)\pi(y)=q_y(x),
\end{align*}
where from the second to last line we have once again used the Markov chain tree theorem.
\end{proof}

\section{Uniform unicycle walk}\label{sec:unicycle-walk}

Athanasiadis \cite{atha} studied a random walk on trees, originally introduced in \cite{markov-tree} as a tool for proving the Markov chain tree theorem. In this model, the root of the tree performs a simple random walk and the entire tree evolves according to the root’s movement. He subsequently proved a conjecture of Propp by determining the spectrum of the Laplacian matrix of the tree walk on complete graphs. In this section, we extend his approach to the total chain of the simple random walk on complete graphs, computing the spectrum of its transition matrix. In addition, we determine the mixing time of the total chain on complete graphs and establish that it exhibits the cutoff phenomenon, meaning that mixing occurs rather abruptly.

\subsubsection*{The graph of unicycles}
Let $G=(V,E)$ be a finite, directed and strongly connected graph. For an unicycle configuration $\eta\in \mathsf{CYC}(G)$ over $G$, we denote the vertices on the unique cycle by $\text{cycle}(U)$ and we recall Definition \ref{def:unic} of the unicycles $\mathsf{U}_G$ of $G$.
Starting from $G$, we construct a new graph whose vertex set is $\mathsf{U}_G$. Two vertices $(x,\rho),(y,\eta)\in\mathsf{U}_G$ are joined by an edge $((x,\rho),(y,\eta))$ whenever the following conditions are satisfied:
\begin{enumerate}
\setlength\itemsep{0em}
\item there exists an edge from $x$ to $y$ in $G$, i.e., $(x,y)\in E$;
\item  the unicycle configurations $\rho$ and $\eta$ coincide everywhere except at $x$, i.e., $\rho\setminus x = \eta \setminus x$;
\item in $\eta$, the arrow at $x$ points to $y$, i.e., $\eta(x)=y$.
\end{enumerate}
We denote the set of all such edges by $E_{\text{cycle}}$.
\begin{defn}[Unicycle Graph]
We call the graph $\mathsf{UCYC}(G)=(\mathsf{U}_G,E_{\text{cycle}})$ the unicycle graph derived from $G$.
\end{defn}
Next, we examine the simple random walk on the unicycle graph  $\mathsf{UCYC}(G)$ of $G$, that we denote it by $(X_n,\rho_n)_{n\in\N}$. This is a special case of a locally Markov walk. 
The following  is straightforward.
\begin{lem}
Consider the simple random walk $(X_n,\rho_n)_{n\in\N}$ on $\mathsf{UCYC}(G)$. Furthermore, let $(Y_n)_{n\in\N}$ be the simple random walk on $G$ and denote its total chain by $(Y_n,\eta_n)_{n\in\N}$. Then $(X_n,\rho_n)_{n\in\N}$ and $(Y_n,\eta_n)_{n\in\N}$ have the same distribution.
\end{lem}
\begin{proof}
By the construction of $\mathsf{UCYC}(G)$,  $(X_n)_{n\in\N}$ is a simple random walk on $G$ and a single step of $(X_n,\rho_n)_{n\in\N}$ corresponds to a step of the total chain of $(X_n)_{n\in\N}$. Thus the claim follows.
\end{proof}
\begin{defn}[Uniform Unicycle Walk]
The simple random walk $(X_n,\rho_n)_{n\in\N}$ on $\mathsf{UCYC}(G)$ is called uniform unicycle walk on $G$.
\end{defn}
\begin{figure}
    \centering
    \begin{tikzpicture}[scale=1.5]
    \begin{scope}[xshift=0cm]
    \filldraw[red] (0,0) circle (1pt);
    \filldraw[black] (1,0) circle (1pt);
    \filldraw[black] (0.5,0.866) circle (1pt);
    \draw[->, >=latex] (0,0) -- (1,0);
    \draw[->, >=latex] (1,0) -- (0,0);
    \draw[->, >=latex] (0.5,0.866) -- (1,0);
    \end{scope}
    \draw[->] (1.5,0.5) -- (2.5,0.5) node[midway, above]{Turning the arrow};
    \begin{scope}[xshift=3cm]
    \filldraw[red] (0,0) circle (1pt);
    \filldraw[black] (1,0) circle (1pt);
    \filldraw[black] (0.5,0.866) circle (1pt);
    \draw[->, >=latex] (0,0) -- (0.5,0.866);
    \draw[->, >=latex] (1,0) -- (0,0);
    \draw[->, >=latex] (0.5,0.866) -- (1,0);
    \end{scope}
    \draw[->] (4.5,0.5) -- (5.5,0.5) node[midway, above]{Moving the particle};
    \begin{scope}[xshift=6cm]
    \filldraw[black] (0,0) circle (1pt);
    \filldraw[black] (1,0) circle (1pt);
    \filldraw[red] (0.5,0.866) circle (1pt);
    \draw[->, >=latex] (0,0) -- (0.5,0.866);
    \draw[->, >=latex] (1,0) -- (0,0);
    \draw[->, >=latex] (0.5,0.866) -- (1,0);
    \end{scope}
\end{tikzpicture}
    \caption{Illustration of a step of the uniform unicycle walk on the complete graph with $3$ vertices.}
    \label{fig:one-step-unicycle-walk}
\end{figure}
\begin{figure}[t]
    \centering
    \begin{tikzpicture}[scale=0.85]
    \begin{scope}[rotate=90]
    \begin{scope}[xshift=0cm]
    \filldraw[red] (0,0) circle (1pt);
    \filldraw[black] (1,0) circle (1pt);
    \filldraw[black] (0.5,0.866) circle (1pt);
    \draw[->, >=latex] (0,0) -- (1,0);
    \draw[->, >=latex] (1,0) -- (0.5,0.866);
    \draw[->, >=latex] (0.5,0.866) -- (0,0);
    \end{scope}
    
    \begin{scope}[xshift=3cm]
    \filldraw[black] (0,0) circle (1pt);
    \filldraw[red] (1,0) circle (1pt);
    \filldraw[black] (0.5,0.866) circle (1pt);
    \draw[->, >=latex] (0,0) -- (1,0);
    \draw[->, >=latex] (1,0) -- (0.5,0.866);
    \draw[->, >=latex] (0.5,0.866) -- (0,0);
    \end{scope}

    \begin{scope}[xshift=6cm]
    \filldraw[black] (0,0) circle (1pt);
    \filldraw[black] (1,0) circle (1pt);
    \filldraw[red] (0.5,0.866) circle (1pt);
    \draw[->, >=latex] (0,0) -- (1,0);
    \draw[->, >=latex] (1,0) -- (0.5,0.866);
    \draw[->, >=latex] (0.5,0.866) -- (0,0);
    \end{scope}

    \begin{scope}[xshift=0cm,yshift=-2cm]
    \filldraw[black] (0,0) circle (1pt);
    \filldraw[black] (1,0) circle (1pt);
    \filldraw[red] (0.5,0.866) circle (1pt);
    \draw[->, >=latex] (0,0) -- (0.5,0.866);
    \draw[->, >=latex] (1,0) -- (0.5,0.866);
    \draw[->, >=latex] (0.5,0.866) -- (0,0);
    \end{scope}
    
    \begin{scope}[xshift=3cm,yshift=-2cm]
    \filldraw[red] (0,0) circle (1pt);
    \filldraw[black] (1,0) circle (1pt);
    \filldraw[black] (0.5,0.866) circle (1pt);
    \draw[->, >=latex] (0,0) -- (1,0);
    \draw[<-, >=latex] (0,0) -- (0.5,0.866);
    \draw[->, >=latex] (1,0) -- (0,0);
    \end{scope}

    \begin{scope}[xshift=6cm,yshift=-2cm]
    \filldraw[black] (0,0) circle (1pt);
    \filldraw[red] (1,0) circle (1pt);
    \filldraw[black] (0.5,0.866) circle (1pt);
    \draw[->, >=latex] (0,0) -- (1,0);
    \draw[->, >=latex] (1,0) -- (0.5,0.866);
    \draw[->, >=latex] (0.5,0.866) -- (1,0);
    \end{scope}

    \begin{scope}[xshift=0cm,yshift=-4cm]
    \filldraw[red] (0,0) circle (1pt);
    \filldraw[black] (1,0) circle (1pt);
    \filldraw[black] (0.5,0.866) circle (1pt);
    \draw[->, >=latex] (0,0) -- (0.5,0.866);
    \draw[->, >=latex] (1,0) -- (0.5,0.866);
    \draw[->, >=latex] (0.5,0.866) -- (0,0);
    \end{scope}
    
    \begin{scope}[xshift=3cm,yshift=-4cm]
    \filldraw[black] (0,0) circle (1pt);
    \filldraw[red] (1,0) circle (1pt);
    \filldraw[black] (0.5,0.866) circle (1pt);
    \draw[->, >=latex] (0,0) -- (1,0);
    \draw[<-, >=latex] (0,0) -- (0.5,0.866);
    \draw[->, >=latex] (1,0) -- (0,0);
    \end{scope}

    \begin{scope}[xshift=6cm,yshift=-4cm]
    \filldraw[black] (0,0) circle (1pt);
    \filldraw[black] (1,0) circle (1pt);
    \filldraw[red] (0.5,0.866) circle (1pt);
    \draw[->, >=latex] (0,0) -- (1,0);
    \draw[->, >=latex] (1,0) -- (0.5,0.866);
    \draw[->, >=latex] (0.5,0.866) -- (1,0);
    \end{scope}

    \begin{scope}[xshift=0cm,yshift=-6cm]
    \filldraw[black] (0,0) circle (1pt);
    \filldraw[red] (1,0) circle (1pt);
    \filldraw[black] (0.5,0.866) circle (1pt);
    \draw[->, >=latex] (0,0) -- (0.5,0.866);
    \draw[->, >=latex] (1,0) -- (0.5,0.866);
    \draw[->, >=latex] (0.5,0.866) -- (1,0);
    \end{scope}
    
    \begin{scope}[xshift=3cm,yshift=-6cm]
    \filldraw[black] (0,0) circle (1pt);
    \filldraw[black] (1,0) circle (1pt);
    \filldraw[red] (0.5,0.866) circle (1pt);
    \draw[->, >=latex] (0,0) -- (0.5,0.866);
    \draw[<-, >=latex] (0,0) -- (0.5,0.866);
    \draw[->, >=latex] (1,0) -- (0,0);
    \end{scope}

    \begin{scope}[xshift=6cm,yshift=-6cm]
    \filldraw[red] (0,0) circle (1pt);
    \filldraw[black] (1,0) circle (1pt);
    \filldraw[black] (0.5,0.866) circle (1pt);
    \draw[->, >=latex] (0,0) -- (1,0);
    \draw[->, >=latex] (1,0) -- (0,0);
    \draw[->, >=latex] (0.5,0.866) -- (1,0);
    \end{scope}

    \begin{scope}[xshift=0cm,yshift=-8cm]
    \filldraw[red] (0,0) circle (1pt);
    \filldraw[black] (1,0) circle (1pt);
    \filldraw[black] (0.5,0.866) circle (1pt);
    \draw[<-, >=latex] (0,0) -- (1,0);
    \draw[<-, >=latex] (1,0) -- (0.5,0.866);
    \draw[<-, >=latex] (0.5,0.866) -- (0,0);
    \end{scope}
    
    \begin{scope}[xshift=3cm,yshift=-8cm]
    \filldraw[black] (0,0) circle (1pt);
    \filldraw[red] (1,0) circle (1pt);
    \filldraw[black] (0.5,0.866) circle (1pt);
    \draw[<-, >=latex] (0,0) -- (1,0);
    \draw[<-, >=latex] (1,0) -- (0.5,0.866);
    \draw[<-, >=latex] (0.5,0.866) -- (0,0);
    \end{scope}

    \begin{scope}[xshift=6cm,yshift=-8cm]
    \filldraw[black] (0,0) circle (1pt);
    \filldraw[black] (1,0) circle (1pt);
    \filldraw[red] (0.5,0.866) circle (1pt);
    \draw[<-, >=latex] (0,0) -- (1,0);
    \draw[<-, >=latex] (1,0) -- (0.5,0.866);
    \draw[<-, >=latex] (0.5,0.866) -- (0,0);
    \end{scope}

    \begin{scope}[xshift=0cm,yshift=-10cm]
    \filldraw[black] (0,0) circle (1pt);
    \filldraw[black] (1,0) circle (1pt);
    \filldraw[red] (0.5,0.866) circle (1pt);
    \draw[->, >=latex] (0,0) -- (0.5,0.866);
    \draw[->, >=latex] (1,0) -- (0.5,0.866);
    \draw[->, >=latex] (0.5,0.866) -- (1,0);
    \end{scope}
    
    \begin{scope}[xshift=3cm,yshift=-10cm]
    \filldraw[red] (0,0) circle (1pt);
    \filldraw[black] (1,0) circle (1pt);
    \filldraw[black] (0.5,0.866) circle (1pt);
    \draw[->, >=latex] (0,0) -- (0.5,0.866);
    \draw[<-, >=latex] (0,0) -- (0.5,0.866);
    \draw[->, >=latex] (1,0) -- (0,0);
    \end{scope}

    \begin{scope}[xshift=6cm,yshift=-10cm]
    \filldraw[black] (0,0) circle (1pt);
    \filldraw[red] (1,0) circle (1pt);
    \filldraw[black] (0.5,0.866) circle (1pt);
    \draw[->, >=latex] (0,0) -- (1,0);
    \draw[->, >=latex] (1,0) -- (0,0);
    \draw[->, >=latex] (0.5,0.866) -- (1,0);
    \end{scope}

    \draw[->, line width = 0.5mm] (1,0.43) -- (3,0.43);
    \draw[->, line width = 0.5mm] (4,0.43) -- (6,0.43);
    \draw [->, line width = 0.5mm] (6,0.52) to [out=150,in=30] (1,0.52);

    \draw[->, line width = 0.5mm] (0.5,-0.1) -- (0.5,-1);
    \draw[->, line width = 0.5mm] (3.5,-0.1) -- (3.5,-1);
    \draw[->, line width = 0.5mm] (6.5,-0.1) -- (6.5,-1);

    \draw[<->, line width = 0.5mm] (0.5,-2.1) -- (0.5,-3);
    \draw[<->, line width = 0.5mm] (3.5,-2.1) -- (3.5,-3);
    \draw[<->, line width = 0.5mm] (6.5,-2.1) -- (6.5,-3);

    \draw[->, line width = 0.5mm] (0.4,-2.1) to [out=180,in=180] (0.4,-5);
    \draw[->, line width = 0.5mm] (3.4,-2.1) to [out=180,in=180] (3.4,-5);
    \draw[->, line width = 0.5mm] (6.4,-2.1) to [out=180,in=180] (6.4,-5);

    \draw[<-, line width = 0.5mm] (2.7,-0.2) -- (0.6,-3);
    \draw[<-, line width = 0.5mm] (5.7,-0.2) -- (3.6,-3);
    \draw[-, line width = 0.5mm] (7.3,0.43) to [out=-90, in=0] (6.6,-3);
    \draw[->, line width = 0.5mm] (7.3,0.43) to [out=90, in=90] (0.5,1.1);

    \draw[<->, line width = 0.5mm] (0.4,-6.1) to [out=180,in=180] (0.4,-9);
    \draw[<->, line width = 0.5mm] (3.4,-6.1) to [out=180,in=180] (3.4,-9);
    \draw[<->, line width = 0.5mm] (6.4,-6.1) to [out=180,in=180] (6.4,-9);

    \draw[->, line width = 0.5mm] (0.5,-6.1) -- (0.5,-7);
    \draw[->, line width = 0.5mm] (3.5,-6.1) -- (3.5,-7);
    \draw[->, line width = 0.5mm] (6.5,-6.1) -- (6.5,-7);

    \draw[<-, line width = 0.5mm] (2.7,-4.2) -- (0.6,-7);
    \draw[<-, line width = 0.5mm] (5.7,-4.2) -- (3.6,-7);

    \draw[<-, line width = 0.5mm] (1,-7.43) -- (3,-7.43);
    \draw[<-, line width = 0.5mm] (4,-7.43) -- (6,-7.43);

    \draw[->, line width = 0.5mm] (0.6,-9) to [out=0,in=0] (0.6,-4.1);
    \draw[->, line width = 0.5mm] (3.6,-9) to [out=0,in=0] (3.6,-4.1);
    \draw[->, line width = 0.5mm] (6.6,-9) to [out=0,in=0] (6.6,-4.1);

    \draw [<-, line width = 0.5mm] (6,-7.52) to [out=-150,in=-30] (1,-7.52);

    \draw [-, line width = 0.5mm] (4.5,-5) to [out=-90,in=180] (6,-7.3);
    \draw [->, line width = 0.5mm] (4.5,-5) to [out=90,in=-90] (0.5,-4.1);
    \end{scope}
\end{tikzpicture}
    \caption{The unicycle graph $\mathsf{UCYC}(K_3)$ together with the uniform unicycle walk. The vertex colored in red indicates the particle location.}
    \label{fig:unicycle-graph-complete-graph}
\end{figure}

\subsubsection*{Uniform unicycle walk on complete graphs}

Here, we consider the uniform unicycle walk on the complete graph $K_m$ with $m$ vertices, where the vertex set is  $\{1,...,m\}$ and the edge set $E(K_m)$ is \textcolor{black}{$\{(i,j):i,j\in\{1,...,m\},i\neq j\}$}. The adjacency matrix $A_m\in\Z^{\mathsf{UCYC}(K_m)\times\mathsf{UCYC}(K_m)}$ of $\mathsf{UCYC}(K_m)$ is defined as
$$A_m((x,\rho),(y,\eta))=\begin{cases}
1,& \text{ if }(x,\rho)\rightarrow(y,\eta)\text{ and }(x,\rho)\neq (y,\eta)\\
0,&\text{ else }
\end{cases},$$
and the Laplacian matrix $L_m\in\Z^{\mathsf{UCYC}(K_m)\times\mathsf{UCYC}(K_m)}$ of $\mathsf{UCYC}(K_m)$ is defined as 
$$L_m((x,\rho),(y,\eta))=\begin{cases}
m-1,&\text{ if }(x,\rho)=(y,\eta)\\
-1,& \text{ if }(x,\rho)\rightarrow(y,\eta)\\
0,&\text{ else}
\end{cases}.$$

\subsubsection*{Eigenvalues of the unicycle Laplacian via path counting}

We compute here the eigenvalues of $L_m$ for $m\in\N$ using the method from \cite{atha}. Note that, for $k\in\N$ and $i,j\in\mathsf{U}_{K_m}$, the entry $A^k_m(i,j)$  counts the number of paths from configuration $i$ to configuration $j$. Consequently, the trace
$\text{tr}A^k_m=\sum_{i\in\mathsf{U}_{K_m}}A^k_m(i,i)$,
counts the number of closed paths in the unicycle graph $\mathsf{UCYC}(K_m)$ of $K_m$.  Alternatively, $A_m^k$ can also be expressed as the sum of the $k$-th powers of the eigenvalues of $A_m$, each weighted by its multiplicity. We focus now on  counting the number of closed walks of a certain length in the unicycle graph $\mathsf{UCYC}(K_m)$ of $K_m$.
Suppose we start the unicycle walk at some configuration $(x,\rho)\in\mathsf{U}_{K_m}$ and end it at the same configuration. Then the roots of the walk trace out a closed walk in $K_m$ starting and ending at $x$. This allows us to derive the number of closed walks in $\mathsf{UCYC}(K_m)$ from the number of closed walks in $K_m$.
For $G=(V,E)$ and $k\in \mathbb{N}$, we denote by $w(G,k)$ the number of closed walks of length $k$ in $G$, and for $S\subseteq V$, we write $G_S$ for the subgraph of $G$ induced by $S$, consisting of the vertices in $S$ and all edges of $G$ connecting pairs of vertices in $S$.

\begin{lem}\label{lem:number-closed-walks}
For any $k\in \N$ we have
\begin{align*}
    w(\textcolor{black}{\mathsf{UCYC}(G)},k)=\sum_{S:\ S\subseteq V}w(G_S,k)\det(\Delta_G\vert_{\textcolor{black}{V\backslash S}}-I).
\end{align*}
\end{lem}
\begin{proof}
Note that a closed walk in $\mathsf{UCYC}(G)$ is uniquely determined by choosing an initial root $x_0$ and a tree rooted at $x_0$ and denoted by $T_0$, as well as a closed walk in $G$ denoted by $x_0\rightarrow x_1 \rightarrow ...\rightarrow x_k$, with $x_k=x_0$. If we write $L:=\max\{1\leq i\leq k:\ x_i=x_0\}$, then the closed walk in $\mathsf{UCYC}(G)$ starts in $(x_0,T_0\cup\{(x_L,x_{L+1})\})$.
Hence, for a closed walk of length $k$ in $G$ given by $x_0\rightarrow x_1 \rightarrow ...\rightarrow x_k$ with $x_k=x_0$, we have to count the number of spanning trees rooted at $x_0$, in order to obtain a closed walk of length $k$ in $\mathsf{UCYC}(G)$. Note that if a vertex $v\in V$ appears in $\{x_0,...,x_k\}$ for the last time with index $L_v$, then in the tree $T_0$ the outgoing edge from $v$ points towards $x_{L_v+1}$. Thus the number of trees $T_0$, which yields a closed walk in $\mathsf{UCYC}(G)$ is given by the number of rooted spanning forests of $G$ with root set $\{x_0,...,x_k\}$. Denote the number of rooted spanning forests of $G$ with root set $S\subseteq V$ by $\tau(G,S)$. 
We have 
\begin{align*}
    w(\mathsf{UCYC}(G),k)=\sum_{S:\  S\subseteq V}\tau(G,S)g(G_S,k),
\end{align*}
where $g(G_S,k)$ is the number of closed walks of length $k$ in $G_S$ visiting all of its vertices. \textcolor{black}{The term $g(G_S,k)$ can be rewritten as}
\begin{align*}
    g(G_S,k)=\sum_{U: \ U\subseteq S}(-1)^{|S\setminus U|}w(G_U,k),
\end{align*}
\textcolor{black}{where we have used the inclusion-exclusion principle for the set of paths in $G_S$ of length $k$ that visit all the vertices in $G_S$.} Plugging this into the formula for $w(\mathsf{UCYC}(G),k)$ and rearranging the  summation order we obtain
\begin{align*}
    w(\mathsf{UCYC}(G),k)=\sum_{S:\ S\subseteq V}w(G_S,k)\sum_{U:\ S\subseteq U\subseteq V}(-1)^{|U\setminus S|}\tau(G,U).
\end{align*}
The matrix tree theorem for computing $\tau(G,U)$ in terms of the Laplacian of $G$ yields
$$\sum_{U:\ S\subseteq U\subseteq V}(-1)^{|U\setminus S|}\tau(G,U)=\sum_{U: \ S\subseteq U\subseteq V}(-1)^{|U\setminus S|}\text{det}(\Delta_G\vert_{V\backslash U}).$$
The claim follows since $\text{det}(\Delta_G\vert_{V\backslash S}-I)$ can be written using the matrix principal minors as
$$\text{det}(\Delta_G\vert_{V\backslash S}-I)=\sum_{U: \ S\subseteq U\subseteq V}(-1)^{|U\setminus S|}\text{det}(\Delta_G\vert_{V\backslash U}).$$
\end{proof}
Consequently, the spectrum of $A_m$ can be derived from the spectra of the induced subgraphs of $K_m$. We denote by
$(K_m)_S:=K_m^S$ the subgraph of $K_m$ induced by a subset $S$ of vertices.
{\color{black} For completeness, we also state below \cite[Lemma 2.1]{atha}, which we will use later on.
\begin{lem}\label{lem:sum_unique}
Assume that for some non-zero complex numbers $(a_i)_{1\leq i\leq r}$ and $(b_j)_{1\leq j \leq s}$ we have 
$\sum_{i=1}^{r}a_i^l=\sum_{j=1}^{s}b_j^l$, for all $l\in\N$.
Then $r=s$ and $(a_i)_{1\leq i\leq r}$ is a permutation of $(b_j)_{1\leq j \leq s}$.
\end{lem}}

\begin{cor}\label{cor:ev-mult}
All nonzero eigenvalues of $A_m$ are given by the nonzero eigenvalues of the induced subgraphs of $K_m$. Furthermore, for $\lambda\neq 0$, the multiplicity $\mathbf{m}_{\mathsf{UCYC}(K_m)}(\lambda)$ of $\lambda$ in $A_m$ is given by
\begin{align*}
    \mathbf{m}_{\mathsf{UCYC}(K_m)}(\lambda)=\sum_{S: \ S\subseteq V}\mathbf{m}_{K_m^S}(\lambda)\det(\Delta_{K_m}|_{\textcolor{black}{V\backslash S}}-I),
\end{align*}
where $\mathbf{m}_{K_m^S}(\lambda)$ denotes the multiplicity of $\lambda$ in the adjacency matrix of  $K_m^S$.
\end{cor}
As a consequence, we determine all the eigenvalues of $L_m$ \textcolor{black}{as follows}.  Let $\mathbf{m}_{L_m}(\cdot)$ denote the multiplicity of an eigenvalue of $L_m$.

\begin{prop}\label{prop:propp}
The Laplacian matrix $L_m$ of the uniform unicycle walk on $K_m$ has eigenvalues $\{-1,0,1,...,m-1\}$ and the corresponding multiplicities are given by 
\begin{align*}
    \mathbf{m}_{L_m}(m-1-i)=\begin{cases}
    i\binom{m}{i+1}(m-1)^{m-i-2}&,1\leq i\leq m-1\\
    m^{m-1}-(m-1)^{m-1}&,i=-1\\
    m^{m-1}(m-2)&,i=0
    \end{cases}.
\end{align*}
\end{prop}
\begin{proof}
Since $L_m=(m-1)I_{|\mathsf{U}_{K_m}|}-A_m$, we can derive the eigenvalues of $L_m$ from those of $A_m$. In order to do so, we use Corollary \ref{cor:ev-mult} together with the fact that the induced subgraphs of $K_m$ are also complete graphs.
Thus for $S\subseteq V$, the eigenvalues of $K_m|_S$ are $|S|-1$ and $-1$ with multiplicities $1$ and  $|S|-1$, respectively. Plugging this into Corollary \ref{cor:ev-mult} completes the proof. To see for example the first case, we remark that for all $\vert S\vert \geq 2$ it holds that
$$\det(\Delta_{K_m}|_{V\backslash S}-I)=(\vert S\vert-1)(m-1)^{m-1-\vert S\vert},$$
thus the value $i$ for $1\leq i \leq m-1$ has multiplicity
$i\binom{m}{i+1}(m-1)^{m-i-2}$ as an eigenvalue of $A_m$.
\end{proof}
Although we are able to determine the spectrum of $L_m$, the eigenvectors and generalized eigenvectors remain out of reach. Concerning the structure of generalized eigenvectors for the uniform unicycle walk on arbitrary finite, strongly connected graphs $G$, we can prove the following result.

\begin{lem}\label{lem:ev-constant}
Let $P$ be the transition matrix of the uniform unicycle walk on a finite, strongly connected graph $G$, and let $f$ be a generalized eigenvector of $P$ associated with an eigenvalue $\lambda \neq 0$. Then, for every vertex $x \in V$ and every oriented tree $T$ rooted at $x$, we have
\begin{align*}
    f(x,T\cup\{(x,y)\})=f(x,T\cup\{(x,z)\}),\quad \text{for } y,z\in N_x.
\end{align*}
\end{lem}
\begin{proof}
Let $\lambda\neq 0$ and let $f$ be a generalized eigenfunction of $P$ with eigenvalue $\lambda$. Then $f\in\text{ran}(P)$.  If $\zeta$ is the unicycle permutation (see Definition \ref{def:perm_matrix}) on $G$, then $PA_{\mathsf{CYC}} ^T$ is a block diagonal matrix, and the blocks are the transition matrices of the local chains. For the uniform unicycle walk, the transition probabilities of the local chains are uniform. Thus for $g\in\text{ran}(PA_{\mathsf{CYC}} ^T)$,  $x\in V$ and all oriented trees $T$ rooted at $x$ we have 
\begin{align*}
    g(x,T\cup\{(x,y)\})=(g,T\cup\{(x,z)\}),
\end{align*}
where $y,z\in N_x$. The claim follows immediately from
$\text{ran}(P)=\text{ran}((P A_{\mathsf{CYC}}^T)A_{\mathsf{CYC}})=\text{ran}(P A_{\mathsf{CYC}}^T)$.
\end{proof}

\subsection{Mixing and cutoff for the uniform unicycle walk on complete graphs}

For two measure $\iota,\zeta$ defined on the same finite set $S$, their total variation distance is 
\begin{align}\label{eq:totvar}
||\iota-\zeta||_{\text{TV}}=\frac{1}{2}\sum_{s\in S}|\iota(s)-\zeta(s)|.
\end{align}
The aim of this part is to prove that the uniform unicycle walk on $K_m$ exhibits cutoff, i.e., for every $\varepsilon \in (0,1)$ we have
$$\lim_{m\rightarrow \infty}\frac{t_{\text{mix}}^{(m)}(\varepsilon)}{t_{\text{mix}}^{(m)}(1-\varepsilon)}=1,$$
where $t_{\text{mix}}^{(m)}(\varepsilon)$ is the mixing time of the uniform unicycle walk $(X^{(m)}_n,\rho^{(m)}_n)_{n\in\N}$ on the complete graph $K_m$ on $m$ vertices,  which is defined as
$$t_{\text{mix}}^{(m)}(\varepsilon):=\min\big\{t\in\N:\max_{(y,\eta)\in\mathsf{U}_G}||{p}^{(t)}((y,\eta),\cdot)-\mu(\cdot)||_{\text{TV}}<\varepsilon\big\}.$$
For the rest of the paper, we allow self-loops in complete graphs, so that the edge set of $K_m$ becomes $\{(i,j): i,j \in \{1,\dots,m\}\}$. This modification simplifies certain calculations, since it enables to describe the cover time of the simple random walk on complete graphs via the classical coupon collector’s problem. We establish this result using the Aldous–Broder backwards algorithm, which generates uniform spanning trees of a graph through a simple random walk (see \cite{aldous-broder}). Recall that the cover time of a random walk is the first time when all states have been visited.

\begin{lem}
Let $G$ be a finite, strongly connected graph, and $(X_n)_{n \in \mathbb{N}}$ be a simple random walk on $G$ with cover time $C$. For $v \in G$, let $l_v(k)$ be the last visit to $v$ before time $k\in \mathbb{N}$. Then the set
$B_{C+1}=\{(X_{l_v(C+1)},X_{l_v(C+1)+1}):v\in G\setminus\{X_{C+1}\}\}$
is distributed as an uniform spanning tree on $G$.
\end{lem}
See \cite{aldous-broder} for a proof.
Denote by $C^{(m)}$ the cover time of the simple random walk on the complete graph $K_m$. By applying the Aldous–Broder backwards algorithm, we derive a bound on the mixing time in terms of the cover time. Denote by $\mu^{(m)}$ the stationary distribution of the uniform unicycle walk on $K_m$ ( i.e.~the simple random walk on the unicycle graph $\mathsf{UCYC}(K_m)$), and define
$$d^{(m)}(t)=\max_{(x,\rho)\in \mathsf{U}_{K_m}}||\mathbb{P}_{(x,\rho)}((X^{(m)}_t,\rho^{(m)}_t)=(\cdot,\cdot))-\mu^{(m)}(\cdot,\cdot)||_{TV}$$
as the distance from the distribution of the uniform unicycle walk at time $t$ to the stationarity of the uniform unicycle walk. As a special case of locally Markov walk, the stationary distribution of the uniform unicycle walk on $K_m$ is described in Proposition \ref{thm:stat-dist}. Bounds on $d^{(m)}$ therefore yield bounds on the mixing time of the uniform unicycle walk.
\begin{lem}\label{lem:unifo-unic-upper}
For the uniform unicycle walk $(X_n^{(m)},\rho_n^{(m)})_{n\in N}$ on $K_m$, the distance between its distribution at time $t$ and its stationary distribution $\mu^{(m)}$ can be expressed as: for $(x,\rho)\in\mathsf{U}_{K_m}$
\begin{align*}
||\mathbb{P}_{(x,\rho)}&((X^{(m)}_t,\rho^{(m)}_t)=(\cdot,\cdot))-\mu^{(m)}(\cdot,\cdot)||_{TV} =\\ &=\mathbb{P}_x(C^{(m)}\geq t)||\mathbb{P}_{(x,\rho)}((X^{(m)}_t,\rho^{(m)}_t)=(\cdot,\cdot)|C^{(m)}\geq t)-\mu^{(m)}(\cdot,\cdot)||_{TV}.
\end{align*}
\end{lem}
\begin{proof}
Suppose that  the uniform unicycle walk on $K_m$ starts at $(x,\rho)\in\mathsf{U}_{K_m}$. Then we have
\begin{align*}
\sum_{(y,\eta)\in\mathsf{U}_{K_m}}|\mathbb{P}_{(x,\rho)}&((X^{(m)}_t,\rho^{(m)}_t)=(y,\eta))-\mu^{(m)}(y,\eta)|\\
&=\sum_{(y,\eta)\in\mathsf{U}_{K_m}}|\mathbb{P}_x(t>C^{(m)})\mathbb{P}_{(x,\rho)}((X^{(m)}_t,\rho^{(m)}_t)=(y,\eta)|t>C^{(m)})\\&
+\mathbb{P}_x(t\leq C^{(m)})\mathbb{P}_{(x,\rho)}((X^{(m)}_t,\rho^{(m)}_t)=(y,\eta)|t\leq C^{(m)})-\mu^{(m)}(y,\eta)|\\
&=\mathbb{P}_x(t\leq C^{(m)})\sum_{(y,\eta)\in\mathsf{U}_{K_m}}|\mathbb{P}_{(x,\rho)}((X^{(m)}_t,\rho^{(m)}_t)=(y,\eta)|t\leq C^{(m)})-\mu^{(m)}(y,\eta)|,
\end{align*}
 where in the last equation above we have used that, in view of the Aldous-Broder backwards algorithm, it holds
$$\mathbb{P}_{(x,\rho)}((X^{(m)}_t,\rho^{(m)}_t)=(y,\eta)|t>C^{(m)})=\mu^{(m)}(y,\eta).$$
Therefore
\begin{align*}
   \mathbb{P}_x\big(t>C^{(m)}\big)&\mathbb{P}_{(x,\rho)}\big((X^{(m)}_t,\rho^{(m)}_t)=(y,\eta)|t>C^{(m)}\big)-\mu^{(m)}(y,\eta)\\
   &=\mathbb{P}_x\big(t>C^{(m)}\big)\mu^{(m)}(y,\eta)-\mu^{(m)}(y,\eta)
   =-\mu^{(m)}(y,\eta)\mathbb{P}_x\big(t\leq C^{(m)}\big),
\end{align*}
  and this completes the proof.
\end{proof}
\begin{cor}\label{cor:unifo-unic-upper}
We have
$d^{(m)}(t)=\max_{x\in K_m}\mathbb{P}_x(C^{(m)}\geq t)$.
\end{cor}
Therefore, to obtain bounds on the mixing time, it suffices to bound the cover time of the simple random walk on $K_m$. This reduces to the classical coupon collector’s problem, for which sharp estimates are available. In particular, we will make use of the tail bounds for the coupon collector’s problem, as given in \cite[pp. 57–63]{coupon}.

\begin{prop}\label{lem:unif-unicyc-mixing-bound}
For the mixing time $t^{(m)}_{\text{mix}}$ of the uniform unicycle walk on $K_m$ we have: for $\varepsilon>0$
$$\Big(1+\frac{\log(1-\varepsilon)}{\log(m)}\Big)m\log(m)\leq t^{(m)}_{\text{mix}}(\varepsilon)\leq\Big(1-\frac{\log(\varepsilon)}{\log(m)}\Big)m\log(m).$$
\end{prop}
\begin{proof}
Let $\varepsilon>0$. Corollary \ref{cor:unifo-unic-upper} together with the tail bound from \cite[pp. 57-63]{coupon} implies 
$$d^{(m)}\Big(\big(1-\frac{\log(\varepsilon)}{\log(m)}\big)m\log(m)\Big)\leq m^{\log(\varepsilon)/\log(m)}=\varepsilon,$$
and thus the upper bound follows.
Also from the tail bound in \cite[pp. 57-63]{coupon} we obtain
$$\max_{x\in K_m}\mathbb{P}_x\Big(C^{(m)}\geq \big(1+\frac{\log(1-\varepsilon)}{\log(m)}\big)m\log(m)\Big)\geq 1-m^{\log(1-\varepsilon)/\log(m)}=\varepsilon,$$
which together with Corollary \ref{cor:unifo-unic-upper} implies
$$d^{(m)}\Big(\big(1+\frac{\log(1-\varepsilon)}{\log(m)}\big)m\log(m)\Big)\geq\varepsilon,$$
which proves the lower bound.
\end{proof}

We can finally prove cutoff for the uniform unicycle walk on $K_m$.

\begin{proof}[Proof of Theorem \ref{thm:cut-off}]
For $\varepsilon>0$ we have
\begin{align*}
\lim_{m\rightarrow\infty}\frac{t_{\text{mix}}^{(m)}(\varepsilon)}{t_{\text{mix}}^{(m)}(1-\varepsilon)}\leq \lim_{m\rightarrow\infty}\frac{\log(m)-\log(\varepsilon)}{\log(m)+\log(\varepsilon)}=1,
\end{align*}
in view of Proposition \ref{lem:unif-unicyc-mixing-bound}. For the lower bound, we also use Proposition \ref{lem:unif-unicyc-mixing-bound} to obtain
\begin{align*}
\lim_{m\rightarrow\infty}\frac{t_{\text{mix}}^{(m)}(\varepsilon)}{t_{\text{mix}}^{(m)}(1-\varepsilon)}\geq \lim_{m\rightarrow\infty}\frac{\log(m)+\log(1-\varepsilon)}{\log(m)-\log(1-\varepsilon)}=1.
\end{align*}
The previous two equations together imply that
$$\lim_{m\rightarrow\infty}\frac{t_{\text{mix}}^{(m)}(\varepsilon)}{t_{\text{mix}}^{(m)}(1-\varepsilon)}=1,$$
and this shows cutoff for the uniform unicycle walk on $K_m$.
\end{proof}

\section{Conclusion and related questions}

We have introduced locally Markov walks on finite graphs, where a particle determines its next move based on its past local actions. To study their long-term behavior — such as convergence to stationarity and ergodicity — we employed the total chain, a Markov chain that records both the particle’s position and the history of local actions. As a concrete example, we analyzed the uniform unicycle walk on complete graphs.

An interesting direction for future work is the study of locally Markov walks on infinite graphs. A natural question is how to characterize recurrence versus transience in terms of the properties of their local memory. While a full general characterization of recurrence is highly nontrivial, it is worthwhile to investigate specific examples of locally Markov walks on infinite graphs. Below, we present a selection of related questions.
\begin{itemize}
\setlength\itemsep{0pt}
    \item Characterize all recurrent locally Markov walks on the integers $\Z$.
    \item The rotor walk with i.i.d.~initial rotors is another example of a locally Markov walk. Prove that the rotor walk on $\Z^d$ is recurrent for $d=2$ and transient for $d\geq 3$.
    \item Prove the recurrence of the $p$-walk on $\Z^2$. It is known that the $p$-walk is recurrent on $\Z$ and admits a scaling limit as shown in \cite{p-walk}.
\end{itemize}
Even on finite graphs, several interesting questions can be addressed. While in this work our focus was on deriving general results applicable to all locally Markov walks under mild assumptions, focusing on specific examples may yield sharper estimates for eigenvalues and mixing times. We present a selection of such problems below.
\begin{itemize}
\setlength\itemsep{0pt}
    \item Examine the mixing behavior of $p$-walks on finite graphs. For which sequences of finite graphs does the $p$-walk exhibit a cutoff phenomenon?
    \item Are there any graph sequences, other than complete graphs, on which the uniform unicycle walk exhibits a cutoff phenomenon?
    \item Describe not only the spectrum but also all the eigenfunctions of the uniform unicycle walk on complete graphs. How does this extend to other families of graphs?
\end{itemize}
\textbf{Funding information.} 
The research of R.~Kaiser was funded by the Austrian Science Fund (FWF) 10.55776/P34129 and the research of E.~Sava-Huss was funded in part by the Austrian Science Fund (FWF) 10.55776/PAT3123425. For open access purposes, the authors have applied a CC BY public copyright license to any author-accepted manuscript version arising from this submission.

\textbf{Acknowledgments.} We are very grateful to the two referees for a very careful reading
of the paper and for the many valuable comments and suggestions, which substantially improved
the paper. In particular, one of the referees suggested a different proof for Proposition \ref{prop:ergodic}, which we include in the paper.

\textsc{Robin Kaiser}, Technische Universität München, Deutschland.\\
\texttt{ro.kaiser@tum.de}

\textsc{Lionel Levine}, Department of Mathematics, Cornell University, USA.\\
\texttt{levine@math.cornell.edu}

\textsc{Ecaterina Sava-Huss}, Universität Innsbruck, Austria.\\
\texttt{Ecaterina.Sava-Huss@uibk.ac.at}

\newpage

\bibliographystyle{alpha}
\bibliography{lit}

\section*{List of notation and symbols}
\begin{tabular}{p{0.15\textwidth} p{0.8\textwidth}}
\textbf{Symbol} & \textbf{Definition}\\
\hline\\
$G=(V,E)$ &finite, directed, strongly connected graph with vertex set $V$ and edge set $E$\\
$N_x$&  neighbours of $x$ in $G$\\
$(X_n)_{n\in \mathbb{N}}$ & locally Markov walk on $G$\\
$\mathsf{G}$&graph induced by the locally Markov walk\\
$\mathsf{E}$&edges induced by the locally Markov walk\\
$(X_n,\rho_n)_{n\in\mathbb{N}}$ & total chain associated to $(X_n)_{n\in \mathbb{N}}$\\
$P$&transition matrix of the total chain\\
$\mathcal{S}$&state space of the total chain\\
$(L^x(k))_{k\in \N}$&local chain at vertex $x\in V$\\
$\mathsf{M}_x$&transition matrix of the local chain $L^x$\\
$\mathsf{m}_x(\cdot,\cdot)$&entries of the matrix  $\mathsf{M}_x$\\
$S_x$&state space of $L^x$ (except in Remark \ref{rem:multipleedge}, where $S_x$ are hidden states at $x$)\\
$q_x$&stationary distribution of the local chain $L^x$\\
$Q$&stochastic matrix indexed over $V\times V$ with row at position $x$ given by $q_x$\\
$\pi$&stationary distribution of $Q$\\
$\rho,\eta$&arrow configurations\\
$\mathsf{TREE}_x(G)$&set of directed spanning trees of $G$ rooted at $x$\\
$\mathsf{CYC}(G)$&set of spanning unicycles of $G$\\
$\psi$&weight of a spanning tree resp. spanning unicycle\\
$\mathsf{U}_G$&set of unicycle configurations\\
$N_t(x)$&number of visits to vertex $x$ until time $t$\\
$B_{\text{loc}}$&block diagonal matrix, blocks are the transition matrices of the local chains\\
$A_{\mathsf{CYC}}$&unicycle permutation matrix\\
$\mathsf{UCYC}(G)$&unicycle graph of $G$\\
$E_{\text{cycle}}$&edges of the unicycle graph\\
$K_m$&complete graph on $m$ vertices\\
$A_m$&adjacency matrix of the unicycle graph of $K_m$\\
$L_m$&Laplacian matrix of the unicycle graph of $K_m$\\
$w(G,k)$&number of closed paths of length $k$ in $G$\\
$\tau(G,U)$&number of rooted spanning forests of $G$ with root set $U$\\
$\Delta_G$&Laplacian of $G$\\
$g(G,k)$&number of closed paths of length $k$ in $G$, visiting every vertex of $G$\\
$G_S$&subgraph induced by the subset $S$ of vertices $V$ of $G$\\
$K_m^S$&subgraph induced by the subset $S$ of vertices of $K_m$\\
$\mathbf{m}_A(\lambda)$&multiplicity of the eigenvalue $\lambda$ of matrix $A$\\
$\mathbf{m}_G(\lambda)$&multiplicity of the eigenvalue $\lambda$ of adjacency matrix of $G$\\
$t_{\text{mix}}^{(m)}$&mixing time of uniform unicycle walk on $K_m$\\
$C^{(m)}$&cover time of simple random walk on $K_m$\\
$\vert\vert\cdot\vert\vert_{TV}$&total variation distance\\
$d^{(m)}(t)$&total variation distance of $t$-step distribution of uniform unicycle walk on $K_m$ to stationarity\\
\end{tabular}

\end{document}